\documentclass[draft]{amsart}
\usepackage{amssymb,latexsym,amsfonts,verbatim,amscd,ifthen,amsmath, graphicx, xypic, amsthm}
\usepackage{color}

\numberwithin{equation}{section}

\theoremstyle{plain}

\newtheorem{theorem}[equation]{Theorem}   
\newtheorem{lemma}[equation]{Lemma}
\newtheorem{proposition}[equation]{Proposition}

\theoremstyle{remark}

\theoremstyle{definition}
\newtheorem{definition}[equation]{Definition}
\newtheorem{remark}[equation]{Remark}

\newtheorem{example}[equation]{Example}

\DeclareMathOperator{\rank}{rank}
\DeclareMathOperator{\im}{Im}

\begin{document}

\renewcommand{\:}{\! :}
\newcommand{\p}{\mathfrak p}
\newcommand{\m}{\mathfrak m}
\newcommand{\e}{\epsilon}
\newcommand{\lra}{\longrightarrow}
\newcommand{\ra}{\rightarrow}
\newcommand{\altref}[1]{{\upshape(\ref{#1})}}
\newcommand{\bfa}{\boldsymbol{a}}
\newcommand{\bfb}{\boldsymbol{b}}
\newcommand{\bfc}{\boldsymbol{c}}
\newcommand{\bfdd}{\boldsymbol{d}}
\newcommand{\bfd}{\boldsymbol{\delta}}
\newcommand{\bfM}{\mathbf M}
\newcommand{\bfI}{\mathbf I}
\newcommand{\bfC}{\mathbf C}
\newcommand{\bfB}{\mathbf B}
\newcommand{\bsfC}{\bold{\mathsf C}}
\newcommand{\bsfT}{\bold{\mathsf T}}
\newcommand{\ol}{\overline}
\newcommand{\twedge}
           {\smash{\overset{\mbox{}_{\circ}}
                           {\wedge}}\thinspace}
\newcommand{\mc}{\mathcal}

\newlength{\wdtha}
\newlength{\wdthb}
\newlength{\wdthc}
\newlength{\wdthd}
\newcommand{\elabel}[1]
           {\label{#1}
            \setlength{\wdtha}{.4\marginparwidth}
            \settowidth{\wdthb}{\tt\small{#1}}
            \addtolength{\wdthb}{\wdtha}
            \raisebox{\baselineskip}
            {\color{red}
             \hspace*{-\wdthb}\tt\small{#1}\hspace{\wdtha}}}

\newcommand{\mlabel}[1]
           {\label{#1}
            \setlength{\wdtha}{\textwidth}
            \setlength{\wdthb}{\wdtha}
            \addtolength{\wdthb}{\marginparsep}
            \addtolength{\wdthb}{\marginparwidth}
            \setlength{\wdthc}{\marginparwidth}
            \setlength{\wdthd}{\marginparsep}
            \addtolength{\wdtha}{2\wdthc}
            \addtolength{\wdtha}{2\marginparsep}
            \setlength{\marginparwidth}{\wdtha}
            \setlength{\marginparsep}{-\wdthb}
            \setlength{\wdtha}{\wdthc}
            \addtolength{\wdtha}{1.1ex}
            \marginpar{\vspace*{-0.3\baselineskip}
                       \tt\small{#1}\\[-0.4\baselineskip]\rule{\wdtha}{.5pt} 
}
            \setlength{\marginparwidth}{\wdthc}
            \setlength{\marginparsep}{\wdthd}  }

\begin{abstract}

Let $R=\Bbbk [x_1,\ldots ,x_m]$ be a polynomial ring in $m$ 
variables over $\Bbbk$ with the standard $\mathbb{Z}^m$ grading and $L$ a multigraded Noetherian $R$-module.  When $\Bbbk$ is a field, Tchernev has an explicit construction of a multigraded free resolution called the T-resolution of $L$ over $R$.  Despite the explicit canonical description, this method uses linear algebraic methods, which makes the structure hard to understand.  This paper gives a combinatorial description for the free modules, making the  T-resolution clearer.  In doing so, we must introduce an ordering on the elements.  This ordering identifies a canonical generating set for the free modules.  This combinatorial construction additionally allows us to define the free modules over $\mathbb{Z}$ instead of a field.  Moreover, this construction gives a combinatorial description for one component of the differential.  An example is computed in the first section to illustrate this new approach.

\end{abstract}

\title{Free modules of a multigraded resolution from simplicial complexes}
\author[A. Beecher]{Amanda Beecher}
\address{Department of Mathematical Sciences\\
         United States Military Academy\\
         West Point, NY 10996}
\email{amanda.beecher@usma.edu}
\keywords{}
\subjclass{}

\maketitle

\section*{Introduction}

Let $R=\Bbbk [x_{1},\ldots,x_{m}]$ be the polynomial ring over a field $\Bbbk$ in $m$ variables with the 
standard $\mathbb{Z}^m$ grading and $L$ a finite $\mathbb{Z}^m$ (multigraded) $R$-module.  Let $$E\stackrel{\Phi}{\lra} G\lra L\lra 0$$ be a minimal finite free multigraded presentation of $L$.  In \cite{Tay}, Diana Taylor created a free resolution for every
 monomial ideal $I$ when $L=R/I$ based on combinatorial properties of the monomial generators of $I$.  Since the entries of $\Phi$ for the minimal presentation are just the monomial generators of $I$, one can describe the Taylor resolution completely from the presentation matrix $\Phi$.  Taylor's construction encouraged the natural generalization of this technique when attempting to construct a free resolution of an arbitrary multigraded module. 
 
In \cite{T}, Tchernev constructs a free resolution for any multigraded module, from a minimal free presentation, which he calls the T-resolution.  The main construction and maps, however, are described from the linear algebra inherited from the matrix $\Phi$.  While this approach does provide a natural generalization since, when $L=R/I$ for a monomial ideal $I$ it is the Taylor resolution, it does not provide the clarity of Taylor's construction.  The three main differences between Taylor's resolution for monomial ideals and the T-resolution for multigraded modules are:
\begin{enumerate}
\item\label{Number1} The T-resolution does not give a description of the free modules in terms of the combinatorial properties of the presentation $\Phi$.  
\item\label{Number2} The T-resolution does not give a generating set for the free modules, which leaves the underlying structure of the free modules opaque.  
\item\label{Number3} The T-resolution is only described over a field, as opposed to the Taylor resolution which is defined over $\mathbb{Z}$.
\end{enumerate}

This purpose of this paper is to describe the free modules in the T-resolution from the viewpoint of matroid theory, thereby solving problems \ref{Number1} and \ref{Number2}, albeit at the expense of ordering the columns of $\Phi$.  This new description of the free modules is defined over $\mathbb{Z}$, which suggests an approach to solving problem \ref{Number3}.

The free modules in the T-resolution are indexed by subsets of the column vectors of $\Phi$.  In our first result, Theorem \ref{C:BetaDim}, we show that the rank of each free module is what some matroid theorists call Crapo's beta invariant, defined by those columns of $\Phi$ indexing that free module.  The discovery of this combinatorial number changed problem \ref{Number1} into a search for a combinatorial object whose dimension gives the beta invariant.  The reduced broken circuit complex is such an object because it is a simplicial complex defined by the columns of $\Phi$ whose only non-zero homology occurs in top dimension and is the beta invariant.

The reduced broken circuit complex was first described by Wilf in \cite{Wilf}, but developed by Whitney, Rota and Brylawski in \cite{Whitney}, \cite{Rota} and \cite{Bry}, respectively.  The homology of the reduced broken circuit complex was shown to have a canonical basis by Bj\"{o}rner in \cite{Bj} and in \cite{Zg}, Ziegler describes this basis explicitly by using an indexing set to describe each basic cycle.  In addition, the construction of the reduced broken circuit complex and its homology are not field dependent.  

Our main results begin with Theorem \ref{T:MultiplicityBasis}, which explicitly states a basis for the underlying vector space of the free modules in the T-resolution.  This basis is described using the same indexing set as in Ziegler's description of the homology facets and is combinatorially defined from the presentation of $\Phi$.  As an interesting consequence, Theorem \ref{T:UniformExample} provides an explicit combinatorially constructed basis for a particular symmetric power of the underlying vector space of the image of $\Phi$ when the presentation matrix $\Phi$ has uniform rank.  In this case, the basis is not the standard monomials of appropriate degree of the generators of the image of $\Phi$.  Instead each basis element is the appropriate number of products of distinct linear forms on the generators, where the coefficients of these forms are the coefficients of the column vectors of $\Phi$ in columns $\{3, \ldots, \rank E\}$.

Theorem \ref{C:canonicalIso} is the culminating result of the paper, which proves that the homology of the reduced broken circuit complex is canonically isomorphic to the underlying vector space of the free modules of the T-resolution.  The map sends a basic cycle defined by Ziegler \cite{Zg} to a basis element from Theorem \ref{T:MultiplicityBasis} which has the same indexing set.  

Our final result, Theorem \ref{T:CommDiagram} uses the canonical isomorphism of Theorem \ref{C:canonicalIso} to show that the long exact sequence in homology for the reduced broken circuit complex commutes with the short exact sequence of the underlying vector spaces when removing the last column of $\Phi$.  This commutativity is particularly interesting because the surjective map of the underlying vector spaces gives rise to one component of the differential of the T-resolution.  In particular, we define the reduced broken circuit complex completely combinatorially from $\Phi$, so that its long exact sequence in homology gives rise to one component of the differential.  This component of the differential can be described over $\mathbb{Z}$ and not just a field, which gives support to the possibility of correcting problem \ref{Number3}, perhaps with a different description of the maps and modules.  If we used this construction on every ordering of our elements, we could theoretically find each component of the differentials.  However, each choice of ordering results in a different generating set and to compile this information into one resolution would result in using the linear algebra properties of the presentation matrix as Tchernev did in the original construction \cite{T}.  Thus, it is unclear how much additional information about the differentials can be found topologically.  

We begin Section \ref{BasisExample} with an example that details the new method to describe the free modules of the T-resolution by defining a basis of the underlying vector space.  This example is easy to follow and only requires some basic knowledge of linear algebra.  We introduce matroid terminology for ease of discussion, but it is not necessary in understanding the construction.  This section also highlights the applications of the main theorems of this paper.  Since this approach creates a bridge between commutative algebra and matroid theory, we introduce notation from both disciplines following mainly from \cite{T} and \cite{Zg}.  For ease of reading, we include matroid theoretic results and definitions used in this paper in Section \ref{ReviewOfMatroids}.  Section 3 describes the beta invariant and states many previously known results about this number.  We also show in Section \ref{DimOfTSpaces} our first main result, that the rank of the free modules in the T-resolution is the beta invariant.  The main object of study is the broken circuit complex, which we describe in Section \ref{BrokenCircuitComplex}.  In Section \ref{TSpaces}, we combine all the matroid theoretic properties of the broken circuit complex to prove our remaining four main results.

I would like to thank my thesis advisor Alexandre Tchernev for many helpful discussions.  I would also like to thank Thomas Zaslavsky who recognized that the formula found for the rank of the free module is a formulation of the beta invariant of matroids.  This suggestion was made to my thesis advisor and is the foundation for this paper.

\section{The method to find a basis for a T-space}\label{BasisExample}

The main result of this paper is a new combinatorially defined basis for the T-spaces and consequently a generating set for the free modules of the T-resolution.  This section gives an example that clearly outlines the construction of the basis as well as highlights some of the difficulties in defining the maps from this basis.  

Given a free presentation of a multigraded module $L$

$$E\stackrel{\Phi}{\lra} G\lra L\lra 0,$$

where $$\Phi = \left( \begin{array}{rrrrr}  x^3 & x^2y& xy^2& x^2 &y^3  \\ x^2 & 2xy & 3y^2 &x &0  \\ \end{array} \right).$$

Consider the set of all column vectors of $\Phi$, for shorthand I will simply use the column number rather than the vector.  For instance, the vectors $$\{ \{ x^3, x^2\}, \{xy^2, 3y^2\}\}=\{1,3\}.$$

Find the collection of all minimal dependent sets; i.e., those dependent sets that are not contained in any other dependent set.  These sets are called circuits.  The circuits of $\Phi$ are: $$\{\{ 1,4\}, \{ 1,2,3\}, \{ 1,2,5\}, \{ 1,3,5\}, \{ 2,3,4\}, \{ 2,3,5\}, \{ 2,4,5\}, \{ 3,4,5\}\}$$  

We then create a lattice, later called the lattice of T-flats, by taking all unions of circuits and ordering them by inclusion.  For this example it is:
{\footnotesize
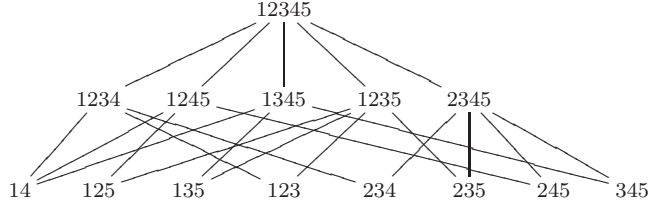
\begin{displaymath}
 \xymatrixcolsep{1pc}
\xymatrix{
 &&& 12345 \ar@{-}[d] \ar@{-}[dr]  \ar@{-}[dl] \ar@{-}[drr] \ar@{-}[dll] & & &\\
& 1234 \ar@{-}[dl]  \ar@{-}[drr] \ar@{-}[drrr]&  1245 \ar@{-}[dll] \ar@{-}[dl] \ar@{-}[drrrr]& 1345  \ar@{-}[dlll] \ar@{-}[dl] \ar@{-}[drrrr]&1235 \ar@{-}[dlll] \ar@{-}[dll] \ar@{-}[dl] \ar@{-}[dr]& 2345\ar@{-}[dl] \ar@{-}[d] \ar@{-}[dr] \ar@{-}[drr]&\\
14& 125  & 135&123&234&235&245&345\\
}
\end{displaymath} 
}

We then remove all members of the lattice that do not contain the smallest element.  Since we have numbered the columns by $\{1,2,3,4,5\}$, we will order the elements $1<2<3<4<5$, although any ordering would work.  However, the remainder of the construction will depend on this ordering, so it must remain fixed throughout.  The poset obtained by removing the subsets that do not contain the element 1 is:
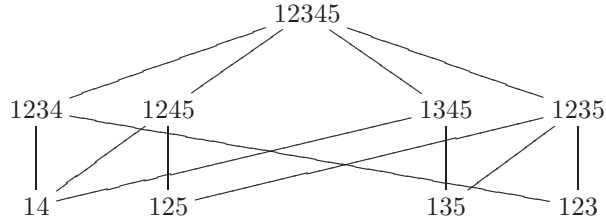
\begin{center}
\begin{displaymath}
\xymatrix{
 && 12345  \ar@{-}[dr]  \ar@{-}[dl] \ar@{-}[drr] \ar@{-}[dll] & & \\
1234 \ar@{-}[d] \ar@{-}[drrrr]&  1245 \ar@{-}[dl] \ar@{-}[d]&& 1345  \ar@{-}[d] \ar@{-}[dlll] &1235 \ar@{-}[d] \ar@{-}[dlll] \ar@{-}[dl] \\
14& 125  && 135&123\\
}
\end{displaymath} 
\end{center}
In order to determine a basis for a subset $A$, say $A=\{1,2,3,4,5\}$, we will find the lexicographically greatest basis defined by the columns of $\Phi$ contained inside the columns indexed by $A$.  In this case, the set $\{4,5\}$ is a basis of the column space of $\Phi$.  So we look at all chains originating in $\{1,2,3,4,5\}$ which remove a subset of $\{4,5\}$, so the edges $\{1,2,3,5\}\lessdot \{1,2,3,4,5\} $ and $\{1,2,3,4\} \lessdot \{1,2,3,4,5\}$ are the only two such chains, since they remove 4 and 5 respectively.  We will then repeat the procedure for the sets $\{1,2,3,5\}$ and $\{1,2,3,4\}$.  The lexicographically greatest basis contained in $\{1,2,3,5\}$ is $\{3,5\}$, so the edges $\{1,2,3\} \lessdot \{1,2,3,5\}$ and $\{1,2,5\} \lessdot \{1,2,3,5\}$ both satisfy the conditions.  For the set $\{1,2,3,4\}$ the lexicographically greatest basis is $\{3,4\}$, so the only chain that satisfies the condition is $ \{1,2,3\} \lessdot \{1,2,3,4\}.$  Notice that the edge $\{1,4\} \lessdot \{1,2,3,4\}$ removes both 2 and 3, and since 2 is not in the basis, this edge does not satisfy the conditions.  The edges satisfying the conditions are labeled on the graph below.

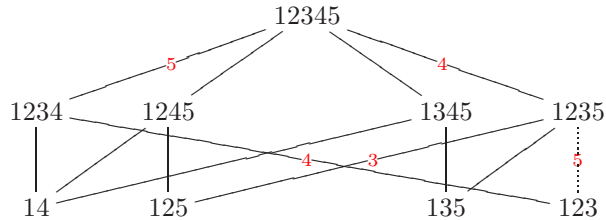
\begin{center}
\begin{displaymath}
\xymatrix{
 && 12345  \ar@{-}[dr]  \ar@{-}[dl] \ar@{-}[drr]|{ \textcolor{red}{ 4 }} \ar@{-}[dll]|{ \textcolor{red}{ 5 } } & & \\
1234 \ar@{-}[d] \ar@{-}[drrrr]|{ \textcolor{red}{ 4 } }&  1245 \ar@{-}[dl] \ar@{-}[d]&& 1345  \ar@{-}[d] \ar@{-}[dlll] &1235 \ar@{.}[d]|{ \textcolor{red}{ 5 } } \ar@{-}[dlll]|{ \textcolor{red}{ 3 } } \ar@{-}[dl] \\
14& 125  && 135&123\\
}
\end{displaymath} 
\end{center}

Beginning at the highest element $\{1,2,3,4,5\}$, we take those labels that appear in descending order and add the minimum element to that set.  Notice that there are three maximal chains with labels, but the one that contains the dotted edge is not decreasing, thus our set only has two elements.  This set is $\{\{1,4,5\}, \{1,3,4\}\}$ and will later be called the $\beta$-nbc basis, given by Ziegler in \cite{Zg} and will index the basis of the T-space associated to $\{1,2,3,4,5\}$.  Since there are two elements, the beta invariant of the matroid given by $\Phi$ is 2.  We reformulate the description given by Ziegler in \cite{Zg} of these sets in order to avoid a certain dualization process.  A more precise statement of this reformulation is given in Remark \ref{R:DualZg}.

For completeness, we have included the $\beta$-nbc basis for each of the T-flats that are not circuits in the following table:

\begin{center}
\begin{tabular}{|l|l|}
\hline

T-flat & $\beta$-nbc basis\\

\hline

$\{1,2,3,4,5\}$ & $\{\{1,4,5\}, \{1,3,4\}\}$\\

\hline

$\{1,2,3,4\}$ & $\{1,4\}$\\

\hline

$\{1,2,4,5\}$ & $\{1,4\}$\\

\hline

$\{1,3,4,5\}$ & $\{1,4\}$\\

\hline

$\{1,2,3,5\}$ & $\{\{1,3\}, \{1,5\}\}$\\

\hline

$\{2,3,4,5\}$ & $\{\{2,4\}, \{2,5\}\}$\\
\hline
\end{tabular}
\end{center}

From this example, the chain that passes through the T-flat $\{1,2,3,4\}$ where only the maximal element is removed from $\{1,2,3,4,5\}$ is $\{1,4,5\}$, and the corresponding $\beta$-nbc basis for $\{1,2,3,4\}$ is the remaining labels from the larger chain $\{1,4\}$.  In general, when we restrict to a smaller T-flat by removing only the largest element, we will find that basis will behave in this predictable way.  This observation illustrates the nice decomposition of maximal chains described in Proposition \ref{P:decreasingchains} into those that pass through the T-flat with the largest element removed and those that do not.   By comparison, if we look at the T-flat $\{1,2,3,5\}$ it also was part of the larger chain for the $\beta$-nbc basis $\{1,3,4\}$, but when we restrict to $\{1,2,3,4\}$ we continue to have two $\beta$-nbc bases with the restricted basis $\{1,3\}$ and the unexpected additional basis given along the dotted edge of $\{1,5\}$.  So not all chains can be nicely described from the larger chains as can be done when we remove the largest element.  In particular, if we consider the T-flat $\{2,3,4,5\}$, the chains do not even belong to this lattice, so there is no clear understanding of how we may describe the $\beta$-nbc basis from the original chains at all.  Understanding how the basis elements transform as we move down the lattice of T-flats is key to understanding the maps of the T-resolution.  This example underscores the difficulty in attempting to duplicate the approach given by removing the largest element to describe all components of the differentials.

Tchernev's construction defines elements of the T-spaces from the chains in the lattice described above.  We show that each chain defined by this process, for example $\{\{1,4,5\}, \{1,3,4\}\}$, index a basis of the underlying vector spaces.  The computation of the specific basis element is shown later in Example \ref{E:TchBasis}.

\section{Review of Matroids}\label{ReviewOfMatroids}
This section introduces basic notation and definitions of matroid theory that will appear throughout this paper.  Matroid theory began as a combination of graph theory and linear algebra, so much of the language will be familiar to those with a background in these disciplines.  We describe the matroid as a simplicial complex, which matroid theorists call the independence complex of a matroid, in order to relate the reduced broken circuit complex, described in Section \ref{BrokenCircuitComplex}, with the original matroid in a natural way.  We emphasize the duality between flats and T-flats, because the construction of Tchernev \cite{T} is explained with T-flats of the matroid given by $\bfM$.  Since most matroid results, including the $\beta$-nbc basis construction by Ziegler \cite{Zg} mentioned in Section 1, are explained in the language of flats, we utilize the duality in order to simplify calculations and maintain the language and notation set forth in Tchernev's original construction \cite{T}.  We also begin a running example in Example \ref{E:UniformMatroid} to illustrate the definitions, but also it creates to an interesting basis for a symmetric power in Theorem \ref{T:UniformExample}.

Given a matrix $\phi$, we define a \emph{matroid} $\bfM$  on a set $S=\{$ the set of column vectors of $\phi\}$ with the independent sets all subsets of $S$ that are linearly independent.  A matroid formed in this way is called \emph{representable} with \emph{representation} $\phi$.   We realize the matroid as a simplicial complex, which matroid theorists call the independence complex, where the vertex set 
$S=\{$ the set of column vectors of $\phi\}$ and the faces are all subsets of $S$ that are 
linearly independent.  In this way, a facet of $\bfM$ is called a \emph{basis} of $\bfM$, since it is a maximal 
linearly independent subset.  
For each $A\subseteq S$, $\dim A +1$ is the \emph{rank} of $A$, denoted $r_{\bfM}(A)=r(A)$.  We call $r(S)$ the rank of the matroid.  Notice it is 
also the rank of the matrix $\phi$.  The \emph{level} of a subset is the number $l(A)=|A|-r(A)-1=|A|-\dim A$.  A minimal non-face (minimal linearly dependent set) 
is called a \emph{circuit} of $\bfM$.  A circuit $A$ such that $|A|=1$ is called a \emph{loop}, which is a missing vertex given my the zero vector.  A matroid $\bfM$ is called \emph{connected} if for every pair of distinct element $x$ and $y$ in $S$ there is a circuit of $\bfM$ containing $x$ and $y$.  Otherwise it is \emph{disconnected}.  If a matroid $\bfM$ contains a loop and $|S|\geq 2$, then it is disconnected.  Indeed, since a loop $\{x\}$ is a circuit, then any circuit containing $y\neq x$ will not contain $x$ by minimality.  In \cite{T}, Tchernev calls a non-empty union of circuits a 
\emph{T-flat} of $\bfM$.  We will denote the collection of all T-flats by $\mc{T}(\bfM)$.  
A subset $Y\subseteq S$ is a \emph{flat} of $\bfM$ if 
for every $x\in S\backslash Y$, $r(Y\cup \{x\})=r(Y)+1$.  The \emph{closure} of $A$, $\overline{A}$, is the unique flat that contains $A$ and has the same rank 
as $A$.  For a representable matroid $\bfM$, $\ol{A}$ is the subset of $S$ which contains all elements in the span of the vectors of $A$ and corresponds to the closure operation in simplicial topology.
  The \emph{dual matroid} $\bfM^*$ of $\bfM$ has vertex set $S$ and the facets of $\bfM ^*$ are all subsets $S-F$ 
where $F$ is a facet of $\bfM$.  We call a loop in $\bfM^*$ a coloop of $\bfM$ and a flat of $\bfM^*$ a dual flat of $\bfM$.  The rank in the dual matroid is the number $r_{\bfM^*}(A)=|A|-r{\bfM}(S)+r{\bfM}(S\backslash A)$.  A set $A\subseteq S$ is a T-flat of a matroid $\bfM$ if and only if $S\backslash
A$ is a proper flat of the dual matroid $\bfM ^*$ (\cite{T}, Definition 2.1.1 and Remark 2.1.2).  

The \emph{lattice of flats} of a matroid $\bfM$ is a ranked lattice orded by inclusion and stratified by the rank 
of the flat in the dual matroid.  So, the lattice of dual flats is ranked by its rank in the dual matroid.  Thus, taking complements of entries in the lattice of dual flats creates a lattice of 
T-flats in the matroid ordered by reverse inclusion and ranked by $$r_{\bfM 
^*}(S\backslash A)=|S|-r_{\bfM}(S)-|A|+r_{\bfM}(A)=l(S)-l(A),$$ for $A$ a T-flat.  This duality allows us to translate the work of matroid theorists on flats to the work of Tchernev \cite{T} which uses T-flats.  The benefit to complementing dual flats into T-flats is to provide a more direct computation of the basis, as shown in Section 1.  Additionally, the T-flats index the multigrading for the free modules in the original construction, so homogenization is more natural from this perspective.

There are two matroid operations that we will use in our induction arguments.  First, the \emph{restriction} of a matroid $\bfM$ to $A$, $\bfM|A$ is defined to be the subcomplex $$\Delta_A
=\{F\in \bfM : F\subseteq A\}.$$  In particular, $\bfM|(S- a)=\bfM\backslash a$ is the subcomplex of $\bfM$ by removing $a$.  Last, the \emph{contraction} of a matroid $\bfM$ on $A$, $\bfM.A$ is the subcomplex $$\Delta_{\bfM.A} =
\{F\subseteq (S-A) : F\cup B \in \bfM, \mbox{ for some } B\subseteq A\}.$$  In particular, $\bfM.(S-a)=\bfM/a=\{F\subseteq S-a: F\cup \{a\} \in \bfM\}$
 is the subcomplex link$(a)$.

\begin{example}\label{E:UniformMatroid}
The matroid $U_{r,n}$ is the uniform matroid on $S$ where $|S|=n$ and $r(\bfM)=r$.  It has the property that every 
r-element subset of $S$ is a basis for $U_{r,n}$.  Therefore, every $r+1$ element subset of $S$ is a circuit of $U_{r,n}$.  
The dual matroid $U_{r,n}^*=U_{n-r, n}$.
\end{example}

If the matroid $\bfM$ is a representable matroid with representation given by\\ $\phi :U\to W$ then $U=U_S$ 
and $U_I$ is the subspaces spanned by $\{e_i\mid i\in I\subseteq S \}$.  The vector space $V_I$ is the subspace of 
$W$ spanned by $\{\phi(e_i)\mid i\in I\}$ and $V=V_S$.  

For a matroid $M$, with $|S|=n$, we can define a linear ordering on the 
vertices by assigning to each vertex an element of the set $S=\{1,2,\ldots , n\}$
and $1<2<\ldots <n.$  We call a matroid with a linear ordering an \emph{ordered 
matroid}.  Many of the constructions we discuss will depend on the ordering we choose.

\section{Dimension of T-spaces}\label{DimOfTSpaces}
Tchernev showed that the rank of the free modules is an invariant of the matroid $\bfM$ from the presentation $\Phi$ \cite[Theorem 3.6]{T}, but does not identify which invariant it is.  Here, Theorem \ref{C:BetaDim} shows that the rank of the free module $T_A (\Phi, S)$ is the beta invariant from the matroid $\bfM|A$.  

The structure of the T-resolution is essentially described in the underlying vector space complex, because there is a canonical procedure to find the T-resolution from the underlying vector space.  The T-resolution is a tensor product between the vector space complex and the ring $R$ by assigning appropriate degrees to each space as outlined in \cite[Definition 4.2]{T}.  Thus the structure of the free resolution lies in the underlying vector space, just like it does for monomial ideals as shown in \cite{Peeva}.  The underlying vector space is obtained from a tensor product of the T-resolution with the field $\Bbbk$, where we consider $\Bbbk$ as the image of sending each variable to the element $1 \in \Bbbk$.  This process essentially removes the grading, so that we can understand the structure of the vector spaces, called \textit{T-spaces} and the maps between them.  The T-space associated with the subset $A$ of columns of $\Phi$ will be denoted $T_A(\phi)$, where $\phi$ is the coefficient matrix of $\Phi$.  Since the free modules $T_A(\Phi, S)=R\otimes_{\Bbbk}T_A(\phi)$ in \cite[Definition 4.2]{T}, then $\dim T_A(\phi)=\rank T_A(\Phi, S)$.

Before we define the beta invariant, we define the \emph{M\"{o}bius function} $\mu$ on a poset $P$ as a map $\mu :P \times P 
\rightarrow \mathbb{Z}$ so that $$\mu (X,X)=1,$$ $$\sum_{X\leq Z\leq Y} 
\mu (X,Z)=0 \mbox{ if } X<Y,$$ and $$\mu (X,Z)=0 \mbox{ otherwise.}$$

\begin{remark}
The M\"{o}bius function for the lattice of T-flats ordered by reverse inclusion is precisely the same as 
the M\"{o}bius function for the lattice of dual flats ordered by inclusion, since we are simply taking complements.
\end{remark}

\begin{definition}\cite{C}
The beta invariant is 
defined as $$\beta (\bfM)=(-1)^{r_{\bfM}(S)} \sum_{x\in \mathfrak{L}(\bfM)} \mu (0,x) 
r_{\bfM}(x),$$ where $\mathfrak{L}(\bfM)$ is the lattice of flats for the matroid 
$\bfM$.
\end{definition}
\begin{lemma}\label{T:DimTspace}
The dimension of a T-space of a T-flat $A$ is given by 
$$ \dim(T_{A})=(-1)^{l(A)} \sum_{B\subseteq A} \mu (A,B) (l(B)+1)$$ where 
$\mu (A,B)$ is the M\"{o}bius function for the partially ordered set of 
connected T-flats for the matroid $\bfM$ under reverse inclusion.
\end{lemma}

\begin{proof}
The T-complex is a resolution for $\ker \phi$.  By \cite[Theorem 5.4.1]{T}, restricting the resolution
 of the $\ker \phi $ to $A$ the 
result will be a resolution of $\ker (\phi |A)$.  This new resolution 
will end at $T_{A}(\phi)$ with all remaining T-spaces of the form $T_{B}(\phi)$ 
where $B\subseteq A$, and $$0=|A|-r^{\bfM 
}_{A}-\sum_{B\subseteq A}(-1)^{l(B)} \dim(T_{B}(\phi)).$$  We can simply rewrite 
this as $$l(A)+1=\sum_{B\subseteq A} (-1)^{l(B)} \dim(T_{B}(\phi)).$$  
Using M\"{o}bius inversion we obtain $$(-1)^{l(A)} \dim T_A(\phi) = \sum_{B\subseteq A} \mu (A,B) (l(B)+1)$$ and our equation is verified.
\end{proof}

The following lemma shows that the dimension of the T-spaces obtained in the previous lemma is equal to the beta invariant.

\begin{lemma}\label{T:BetaDim}
Given a representable matroid $\bfM$ on a finite set $S$ with representation 
$\phi,$ $$\dim T_{S}(\phi )=\beta (\bfM ^*).$$
\end{lemma}

\begin{proof}
$\displaystyle \beta (\bfM ^*)=(-1)^{r_{\bfM ^*}(S)} \sum_{x\in \mathfrak{L}(\bfM ^*)} 
\mu (\emptyset,x) r_{\bfM ^*}(x)$, where $\mathfrak{L}(\bfM ^*)$ is the lattice of dual flats of $\bfM$ or in other words the lattice of flats in the dual matroid.  This yields\\
{\small\[ 
\begin{array}{ll}
\beta (\bfM^*) & \displaystyle
=(-1)^{l(S)+1}\sum_{x\in \mathfrak{L}(\bfM ^*)} \mu (\emptyset,x) \left[(l(S)+1)-(l(S\backslash x)+1)\right]\\
& \displaystyle =(-1)^{l(S) +1} (l(S)+1) \sum_{x\in \mathfrak{L}(\bfM ^*)} \mu (\emptyset,x)+(-1)^{l(S)} \sum_{x\in \mathfrak{L}(\bfM ^*)} \mu (\emptyset,x) (l(S\backslash x)+1)\\
& \displaystyle =(-1)^{l(S)} \sum_{x\in \mathfrak{L}(\bfM ^*)} \mu (\emptyset,x) (l(S\backslash 
x)+1)\\
& \displaystyle =(-1)^{l(S)} \sum_{B \in \mc{T}( \bfM)} \mu (S,B) 
(l(B)+1)\\
& \displaystyle = \dim T_S(\phi) \mbox{ by Lemma \ref{T:DimTspace}}.
\end{array}
\] }
\end{proof}

\begin{example}\label{uniform}
If $\bfM=U_{r,n}$, then $\dim T_S(\phi)=\binom{n-2}{r-1}=\beta (U_{r,n})=\beta (U_{r,n}^*)$ where the first equality follows by \cite[Example 2.2.10]{T}, the second is \cite[Proposition 7]{C}, and the third is \cite[Theorem IV]{C}.
\end{example}

\begin{theorem}\label{C:BetaDim}
Given a representable matroid $\bfM$ on a finite set $S$ with representation 
$\phi$ and a T-flat $A$ of $\bfM$, then $\dim T_{A}(\phi )=\beta ((\bfM |A)^*)$.
\end{theorem}

\begin{proof}
This corollary is immediate from Theorem \ref{T:BetaDim} since $T_A(\phi)=T_A(\phi|A)$ \cite[Theorem 5.4.1]{T} where $\phi |A$ is the submatrix of $\phi$ whose 
columns are indexed by $A$.  So for each $A$, we replace $\bfM$ by $\bfM |A$ and obtain our result.
\end{proof}

We end this section with properties of the beta invariant given in \cite{C}.
\begin{enumerate}
\item The number $\beta (\bfM) \geq 0$ for every matroid $\bfM$.
\item  $\beta (\bfM)=0$ precisely when $\bfM$ is disconnected.
\item For any matroid $\bfM$ with dual $\bfM ^*, \ \beta (\bfM)=\beta (\bfM 
^*)$
\item If $a$ is not a loop or a coloop, then $\beta (M)=\beta (M/a)+\beta 
(M\backslash a).$
\end{enumerate}
These properties will be used later to prove the main theorems.

\section{Broken Circuit Complex}\label{BrokenCircuitComplex}
This section sets the foundation for the key object of study, the broken circuit complex.  We begin with a review of the key properties of this complex, focusing on the construction of Ziegler \cite{Zg} of the $\beta$-nbc basis, which indexes the basic cycles of the reduced homology of the broken circuit complex.  In the next section, we show that the $\beta$-nbc basis also indexes a basis of a T-space, thereby defining the canonical isomorphism of Theorem \ref{C:canonicalIso}.  We also use the $\beta$-nbc basis to show that the broken circuit complex mimics the decomposition under restriction and contraction inherent in the T-spaces in Proposition \ref{P:decreasingchains}.  This allows us to realize the maps of the long exact sequence of homology in Lemma \ref{L:HomologyMaps} that will be shown to give one component of the differential in our final theorem, Theorem \ref{T:CommDiagram}.

In an ordered matroid a \emph{broken circuit} is obtained from a circuit by deleting its smallest 
element, here we will call it 1.  The family of all subsets of $S$ that contain no broken circuits is called 
the \emph{broken circuit complex} of $\bfM$, written $BC(\bfM)$ and the family of all subsets of $S\backslash 1$ that contain no broken circuits is called
the \emph{reduced broken circuit complex} of $\bfM$, written $\ol{BC}(\bfM)$.  

\begin{example}
The circuits of $U_{r,n}$ are all $r+1$ element subsets.  The broken circuits of $U_{r,n}$ are all $r$ element 
subsets except those that contain $1$.  Thus $BC(U_{r,n})$ is a complex with faces all $r$ simplices that contain 1 and $\ol{BC}(U_{r,n})$ is a complex whose faces are all $r-1$ simplicies of $\{2,\ldots n\}$, which is isomorphic to the simplicial complex defined by $U_{r-1, n-1}$. 
\end{example}

In \cite{Bry}, it is shown that for $\bfM$ an ordered matroid of rank $r$ that 
\begin{enumerate}
\item
$\ol{BC}(\bfM)\subseteq BC(\bfM)\subseteq \bfM$.
\item
$BC(\bfM)$ is a pure (r-1)-dimensional complex of $\bfM$.
\item
$BC(\bfM)$ is a cone over $\ol{BC}(\bfM)$ with apex 1, whose facets we call nbc-basis.
\item
$\ol{BC}(\bfM)$ is a pure (r-2)-dimensional subcomplex of $\bfM$.
\end{enumerate}

 In \cite[Theorem 7.8.2 and 7.7.2]{Bj}, Bj\"{o}rner shows that 
the only nonzero homology of reduced broken circuit complex occurs in the top dimension and establishes the existence of a canonical set of basic cycles indexed by certain facets of the simplicial complex.  These facets are called homology facets and depend on the ordering of the vertices.  

In \cite[Theorem 1.4 and 2.4]{Zg}, Ziegler describes the homology facets for the reduced broken circuit complex, denoted $\beta \mbox{nbc}(\bfM)$ in two ways.  First, $B\in \beta \mbox{nbc}(\bfM)$ if and only if
 $B$ is a facet of $BC(\bfM)$ and has no non-trivial internal activity.  That is to say that, for each $b\in B, \ b\neq \min (c^*(B,b))$ where $c^*(B,b)$ is the unique
 circuit in the dual matroid $\bfM^*$ containing $b$ and contained in $(S\backslash B)\cup \{b\}$ \cite[Definition 1.3]{Zg}.
Also, it is shown that
 $B\in \beta \mbox{nbc}(\bfM)$ if and only if a decreasing chain of a certain E-L shelling whose second entry is $b_i\in B$ 
for all covers in the chain of flats.  In other words, each $b_i=\min (G\backslash F)$ for a cover $F\lessdot G$ where $G$ does not intersect the lexicographically least basis of $\bfM/(F\cup \{1\})$ \cite[Theorem 1.4 and 2.4]{Zg}.

The description of the $\beta$-nbc basis given in Section 1 is the dualized version of Ziegler's description given here.  The construction explained above is given using the lattice of flats of a matroid, whereas Tchernev's construction \cite{T} is given in terms of the lattice of T-flats.  In Section 2, we observed that the lattice of T-flats of a matroid is the complement of the lattice of dual flats.  The following remark explains the translation of Ziegler's lattice of dual flats into information about the lattice of T-flats.
\begin{remark} \label{R:DualZg} \hfill

\begin{enumerate}
\item
The set $B\in \beta \mbox{nbc} (\bfM^*)$ if and only if $S\backslash B$ is a basis of $\bfM$ not containing $1$ so that for each $b\in B$, $b$ is not the minimal element of the circuit contained in the dependent set $(S\backslash B)\cup \{b\}$ of $\bfM$.
\item
A flat $F$ is an element of the reduced lattice of dual flats by removing all elements that do not contain the smallest element $1$ if and only if $S\backslash F$ is a T-flat that contains the element $1$.
\item
A cover $F\lessdot G$ in the reduced lattice of dual flats of the dual matroid if and only if $S\backslash G \lessdot S\backslash F$ is a cover in the lattice of T-flats.
\item
The lexicographically least basis of $\bfM^*$ inside $\bfM^*/(F\cup \{1\})=\bfM \backslash (F\cup \{1\})$ is the complement to the lexicographically greatest basis of $\bfM$ inside $S\backslash (F\cup \{1\})$.  
\item
The lexicographically least basis of $\bfM^*$ inside $\bfM^*/(F\cup \{1\})=\bfM \backslash (F\cup \{1\})$ is inside the T-flat $S\backslash F$.
\item
Since $G$ does not intersect $\bfM^*/(F\cup \{1\})\subseteq S\backslash F$, it is enough to say that $(G\backslash F)$ does not intersect $\bfM^*/(F\cup \{1\})$.  This is equivalent to saying that $(G\backslash F)$ is contained in the complement of this set, which is the lexicographically greatest basis of $\bfM$ inside $S\backslash (F\cup \{1\})$.
\end{enumerate}
\end{remark}

\begin{example}
 For the uniform matroid $U_{r,n}$, we have $$\beta \mbox{nbc}(U_{r,n}^*)=\{X\subseteq S\mid 
|X|=n-r, 1\in X \mbox{ but } 2\notin X\}$$. 
\end{example}

We use the properties of the beta invariant stated earlier with the reinterpretation of Ziegler's construction to show how the chains representing $\beta \mbox{nbc}(\bfM)$ decompose under restriction and contraction.

\begin{proposition} \label{P:decreasingchains}
Consider a linear ordering on $S$ of a connected matroid $\bfM$ and let $a=\max(S)$.
If $S\backslash a$ is not a T-flat, then $\bfM /a$ has the same lattice of T-flats as $\bfM$, so the decreasing chains remain the same.  If $S\backslash a$ is a T-flat of $\bfM$ then the decreasing chains of $\bfM \backslash a$ are precisely those chains of $\bfM$ that passed through $S\backslash a$ with $\{a\}$ removed.  Also, if $S\backslash a$ is a T-flat of $\bfM$ then the decreasing chains of $\bfM / a$ are precisely those chains of $\bfM$ that did not pass through $S \backslash a$.
\end{proposition}
\begin{proof}
If $S\backslash a$ is not a T-flat the $\bfM /a$ and has the same lattice of T-flats as $\bfM$, except the T-flats that contained the element $\{a\}$ are renamed without it.  The labels of the decreasing chains remain unchanged, because $a=\max (S)$ so was not a label except possibly on the first cover, but the labels are the minimum element from the removed set and we assumed that $S\backslash a$ was not a T-flat, so $a$ is not a label.  Now, we will suppose that $S\backslash a$ is a T-flat for the remainder of the argument.  

The lattice of T-flats for $\bfM \backslash a$ is the interval $[\emptyset, S\backslash a]$ and the T-flats remain unchanged.  Since these T-flats never contained the element $\{a\}$, it cannot be in the lexicographically greatest basis nor the edge label.  So, therefore the labels remain unchanged and the decreasing chains of $\bfM \backslash a$ are precisely those from $\bfM$ that pass through $S\backslash a$.  Since we have removed the last cover, we must therefore remove the last label which was $\{a\}$.  

The lattice of T-flats for $\bfM /a$ is obtained from the lattice of T-flats $\bfM$ by intersecting all T-flats with $S\backslash a$ and identifying any redundant labels.  In other words, we take all T-flats that contain $\{a\}$ and remove the element $a$ and then identify common labels.  So if $J \lessdot I$ is an edge of a decreasing label in $\bfM$ so that $a\in J$, then $a\in$ lexicographically greatest basis of $I$, call it $B$.  Indeed, since $S\backslash a$ is a T-flat, then $\{a\}$ is independent, so if $B'$ is the lexicographically greatest basis, then there is a $b<a$ so that $(B'\backslash b)\cup \{a\}$ is a basis and is lexicographically larger than $B'$.  Thus, the lexicographically greatest basis of $I\backslash a$ is $B\backslash a$.  Since $(I\backslash a)\backslash (J\backslash a)=I\backslash J \subseteq B$ and $a\in J$ then $(I\backslash a)\backslash (J\backslash a)\subseteq B\backslash a$, so the edge remains a decreasing chain with the same label.  

Suppose that $a\in I$ but $a\notin J$.  If $I\backslash J=\{a\}$, then $I=S$ and $J=S\backslash a$, else it could not be a decreasing chain, because $a=\max(S)$ and this is the previous case.  So we must have that there is more than one element in $I\backslash J$.  Thus in $\bfM /a$, $J \lessdot I\backslash a$ is an edge.  If we let $B$ be the lexicographically greatest basis in $I$, then $a\in B$ by the earlier argument and the lexicographically greatest basis of $I\backslash a$ is $B\backslash a$.  So $(I\backslash a)\backslash J\subseteq B\backslash a$.  Also, $\min(I\backslash J)=\min((I\backslash a)\backslash J)$ since $a\neq \min (I\backslash J)$, so the label remains the same.

\end{proof}

The basic cycles of the reduced homology of the reduced broken circuit complex decompose under restriction and contraction of the largest element just as the beta invariant does as shown in \cite{C}; i.e., $$\beta(\bfM)=\beta(\bfM /a)+\beta(\bfM \backslash a).$$  Moreover, $|\beta \mbox{nbc}(\bfM)|=\beta(\bfM)$  the beta invariant of a matroid $\bfM$ \cite{C}.

When $\bfM$ is a connected matroid and $a$ is the largest element in the ordering on $S$ the simplicial complex $\ol{BC}(\bfM \backslash a)$ 
is a subcomplex of $\ol{BC} (\bfM)$ consisting of all faces $\ol{BC} (\bfM)$ that do not contain $a$ by 
Proposition \ref{P:decreasingchains}.  
This induces a canonical exact sequence of reduced chain complexes $$0\to \widetilde{C}(\ol{BC} 
(\bfM\backslash a))\to \widetilde{C}(\ol{BC} (\bfM))\to \widetilde{C}(\ol{BC} (\bfM/ a))[-1]\to 0,$$  
which induces the following long exact sequence in reduced homology: 
\[
\begin{CD}
0 \rightarrow \widetilde{H}_{l-1}(\ol{BC}(\bfM\backslash a);\Bbbk) \stackrel{\epsilon}{\lra}
@.
\widetilde{H}_{l-1}(\ol{BC}(\bfM);\Bbbk)\stackrel{\delta}{\lra}
@.
\widetilde{H}_{l-2} (\ol{BC}(\bfM/ a);\Bbbk) \rightarrow 0 
\end{CD}
\]
where $l$ is the level of $\bfM$.

We describe the maps explicitly from Ziegler's description of the basic cycles in reduced homology. More precisly, for each $B\in \beta \mbox{nbc}(\bfM)$,  $$\overline{\sigma}_B=\sum_{i_1=0}^{e_1} \sum_{i_2=0}^{e_2}\ldots \sum_{i_n=0}^{e_n} (-1)^{i_1+\ldots +i_n}
[(A_1-a_{i_1}^1),(A_2-a_{i_2}^2), \ldots , (A_n-a_{i_n}^n)]$$ is
the associated basic cycle in the reduced homology of the reduced broken circuit complex given in \cite{Bj} and \cite{Zg}.  The sets $A_i=\{i\} \cup \varphi^{-1}\{i\}$ are written in increasing 
order where $\varphi : B\to S$ is the map defined as $\varphi (b)=\min (c^*(B,b))$ and  $e_i=|A_i|-1$.  

\begin{lemma} \label{L:HomologyMaps}
For $\ol{\sigma}_B\in \widetilde{H}_{l-1}(\ol{BC}(\bfM\backslash a);\Bbbk)$, $$\epsilon (\ol{\sigma}_B)=\ol{\sigma}_B 
\in \widetilde{H}_{l-1}(\ol{BC}(\bfM);\Bbbk)$$ and the map $\delta$ evaluated on a cycle
$\ol{\sigma} _B \in \widetilde{H}_{l-1}(\ol{BC}(\bfM);\Bbbk)$ yields
$$\delta(\ol{\sigma} _B)=\left\{ \begin{array}{ll} \ol{\sigma}_{B \backslash a} & \mbox{if } a\in 
B\\
0 & \mbox{ otherwise.}\\
\end{array}
\right.$$  

\end{lemma}

\begin{proof}
First we show that the map $\varphi \: B\to S$ where $B\in \beta \mbox{nbc}(\bfM)$ 
is preserved under restriction if 
$a\notin B$ and under contraction if $a\in B$.  

If $a\notin B$, then $c^*_{\bfM\backslash a}(B,b)=c^*_{\bfM}(B,b)\backslash a$.  Also, we know that $a\neq \min c^*(B,b)$ since $\min c^*(B,b)<b<a$.  
Therefore, the map $\varphi$ is preserved under restriction.  

For the contraction, we use a similar approach.  Since $a\in B$, then for all $b\in B$ and $b\neq a$, we have that $c^*(B,b)\subseteq (S\backslash B)\cup \{b\}$ so that $a\notin c^*(B,b)$ and thus $c^*(B,b)$ is a circuit 
in $\bfM \backslash a$.  Therefore, the map $\varphi$ is also preserved under contraction.
 
The map $\epsilon$ is induced from the inclusion.  Moreover, since we have that $B\in \beta \mbox{nbc} 
(\bfM \backslash a)$ is also in $\beta \mbox{nbc}(\bfM)$ and  $\varphi$
 is preserved under restriction on the maximal element, 
we see the subsets $A_i$ will not change so that the basis element given by $B$ in 
$\widetilde{H}_{l-1}(\ol{BC}(\bfM\backslash a);\Bbbk)$ is the same as the element given by $B$ in 
$\widetilde{H}_{l-1}(\ol{BC}(\bfM);\Bbbk).$  

The map $\delta$ is induced from the projection onto $\ol{BC}(\bfM/a)*a$ then removing $a$.  Moreover, 
we have that 

\noindent $\delta(\overline{\sigma}_B)$

\noindent \quad $\displaystyle = {\tiny \sum_{i_1=0}^{e_1} \ldots \sum_{i_t=0}^{e_t-1}\ldots \sum_{i_n=0}^{e_n} (-1)^{i_1+\ldots +i_n}
[(A_1-a_{i_1}^1),\ldots ,((A_t\backslash a)-a_{i_t}^t), \ldots , (A_n-a_{i_n}^n)}]$ 

\noindent \quad $\displaystyle =\overline{\sigma}_{B\backslash a}$

 because 
$\varphi$ preserves the $A_i$ under contraction and removing $e_t$ does not change any of the coefficients.  
\end{proof}

\section{Main Theorems}\label{TSpaces}

The problem of finding a suitable combinatorial model for the vector spaces $T_A(\phi)$
has been changed in Section \ref{DimOfTSpaces} to finding a combinatorial model for the beta invariant.  The previous section showed that the reduced homology of the reduced broken circuit complex has the requisite properties and is described completely from the matroid $\bfM$.  We begin by defining a basis for the multiplicity spaces, which is a component of the dual of the T-spaces.  An interesting basis for a symmetric power of the multiplicity space is described from this new basis in Theorem \ref{T:UniformExample}.  We also explicitly define the canonical isomorphism between the homology of the reduced broken circuit complex and the T-spaces in Theorem \ref{C:canonicalIso}.  Finally, Theorem \ref{T:CommDiagram} proves that one component of the differential is the topological map defined in Section \ref{BrokenCircuitComplex}.

\begin{definition}\label{T:posDim}
Each $B\in \beta \mbox{nbc}(\bfM^*)$ determines canonically an element $x_B$ of the multiplicity space $S_S(\phi)$ defined as follows.
 
Let $J_0\subset J_1\subset \ldots \subset J_{l(S)}$ be a chain in the lattice of T-flats of $\bfM$ corresponding to $B$.  Set $$x_B=\phi(u_1)\cdot \phi(u_2)\cdot \ldots \cdot \phi(u_{l(S)}),$$ where each $\displaystyle u_i=e_p +\sum_{j<p}x_j e_j$ is 
the unique element of  $U_{J_i\backslash J_{i-1}},$ 
where $p$ is the largest element of $J_i\backslash J_{i-1}$ in the fixed ordering so that $\phi(u_i) 
\in V_{J_{i}\backslash J_{i-1}}\cap V_{J_{i-1}}$  \cite[Remark 5.2.3]{T}.
\end{definition} 

\begin{theorem}\label{T:MultiplicityBasis}
The collection $\{x_B\mid B\in \beta \mbox{nbc}(\bfM^*)\}$ is a basis for the multiplicity space $S_S(\phi)$.
\end{theorem}

\begin{proof}
If $\bfM$ is disconnected, then $S_S(\phi)=0$ by \cite[Theorem 5.3.5]{T}, $\beta (\bfM^*)=0$ by \cite{C} and $\beta \mbox{nbc}(\bfM^*)=\emptyset.$  Therefore, we may assume that $\bfM$ is a connected matroid.
We will use induction on $|S|$.
Suppose $ |S|=1$, then $S=\{a\}$ is dependent or independent.  If $S$ is independent, then $S_S(\phi)= 0$ by \cite[Definition
 2.2.3]{T}.  Also, $S$ is a loop in $\bfM ^*$, so $r(\bfM^*)=0 $ and $\beta \mbox{nbc} (\bfM^*)=\emptyset$.  
Therefore, there is nothing to show.  If $S$ is dependent then $S$ is independent in $\bfM^*$, 
so $r(\bfM^*)=1$ and has no circuits and therefore no broken circuits.  Therefore, $\beta \mbox{nbc}(\bfM^*)=\{\{a\}\}$.  
Thus $x_{\{a\}} $ is the product of zero factors, hence $x_{\{a\}} =1\in \Bbbk =S_S (\phi )$ by Theorem \ref{T:posDim} and \cite[Definition 2.2.3 (c)]{T}.

Now, suppose $n=|S|>1$.  Let $a=\max (S)$.  By our inductive hypothesis and Proposition \ref{P:decreasingchains}, $\{x_B\mid B\in 
\beta \mbox{nbc}(\bfM^*\backslash a)
=\beta \mbox{nbc}((\bfM/a)^*)\}$ is a basis for $S_{S\backslash a}(\phi/a)$ and $\{x_B\mid B\in \beta \mbox{nbc}(\bfM^*/a)=
\beta \mbox{nbc}((\bfM\backslash a)^*)\}$ is a basis for $S_{S\backslash a}(\phi\backslash a)$.

We have the short exact sequence of vector spaces from \cite[Theorem 5.4.5]{T}
\[
\begin{CD}
0\rightarrow S_{S\backslash a}(\phi)\otimes V_{a}(\phi)
@> \nu >>
 S_S(\phi)
@>\pi_{S/a} >>
S_{S\backslash a}(\phi .S\backslash a)\rightarrow 0
\end{CD}
\]
When the basis $B\in \beta \mbox{nbc}(\bfM^*)$ and $a\notin B$, then $\pi_{S/a}(x_B)=x_B$ and $B\in \beta \mbox{nbc}(\bfM^*
\backslash a)$.  Let $x_B$ be a basis element in $S_{S\backslash a}(\phi)$ where 
$B\in \beta \mbox{nbc}(\bfM^*/a)$ and $\phi(e_a)$ the linear form in $V_a (\phi)$.  
Then $$\nu (x_B\otimes \phi(e_a))=\phi(e_a)\cdot x_B=x_{B\cup a}\in S_S(\phi),$$  because $e_a=u_{l(S)}$ in $U_{S\backslash 
(S\backslash a)}=U_a$.  
Thus, $\{x_B \mid B\in \beta \mbox{nbc}(\bfM^*)\}$ gives a spanning set of 
$S_S(\phi)$.  However, there are $\beta(\bfM^*)$ elements in this spanning set and $\dim_\Bbbk S_S(\phi)=\dim _\Bbbk 
T_S(\phi)=\beta(\bfM)=\beta(\bfM^*)$, so $\{x_B \mid B\in \beta \mbox{nbc}(\bfM^*)\}$ is a basis of $S_S(\phi)$.
\end{proof}

Theorem \ref{T:MultiplicityBasis} leads to an interesting description of a basis for the symmetric power of $V$ because \cite[Remark 2.2.10]{T} showed that $S_S(\phi)=S_{n-r-1}V$, the $(n-r-1)$-th symmetric power of $V=\im \phi$ when $\phi$ is a representation of the uniform matroid $U_{r,n}$.

\begin{theorem}\label{T:UniformExample}
If $\phi$ is a representation 
of $U_{r,n}$, then the set $$\{ \phi (e_B)\mid |B|=n-r-1 \mbox{ and } 1,2 \notin B \}$$ is a basis for $S_S(\phi)=S_{n-r-1}V$, the $(n-r-1)^{st}$ symmetric power of $V=\im \phi$.  In this context, $\phi(e_B)=\phi(e_{b_1})\cdot \phi(e_{b_2})\cdot \ldots \cdot \phi(e_{b_{n-r-1}})$, where $b_i\in B$.
\end{theorem}

Essentially, this says that we can define a basis of $S_{n-r-1}V$ by taking all $(n-r-1)$ products of the linear forms in a basis of $V$ defined by the columns $\{3, \ldots , n\}$ of $\phi$.  The standard basis of $S_{n-r-1}V$ are all monomials of degree $n-r-1$ in a basis of $V$.  This is a curiously unusual basis arriving in this context. 

As in \cite[Proof of Theorem 6.1]{T}, and using the fixed linear ordering on $S$,  
we identify each subset of $S$ with the 
increasing sequence of 
its elements. For a sequence 
$K=(a_1,\dots,a_q)$, we set  
\[
e_K = 
e_{a_1}\wedge \dots 
       \wedge e_{a_q} \in \wedge^q U
\quad\text{and}\quad  
v_K = 
\phi(e_{a_1})\wedge\dots
             \wedge\phi(e_{a_q})\in
             \wedge^q V.
\] 
If $K$ is a subset of $S$ and 
the elements $\phi(e_{a_1}), \dots, \phi(e_{a_q})$ 
form a basis of $V_K$, we write $e^*_K$ and $v^*_K$ 
for the unique elements of $\wedge^q U_K^*$ and 
$\wedge^q V^*_K$ respectively, such that 
$e^*_K(e_K)=1$ and $v^*_K(v_K)=1$. 
In particular, if $B$ is a basis of $\bfM^*$ then $v_{S\backslash B}$ is a basis 
of $\wedge^{r(S)} V_S$ and $v^*_{S\backslash B}$ 
is the dual basis of $\wedge^{r(S)} V^*_S$.

\begin{lemma}\label{T:canonicalIso}
Let $\bfM$ be the ordered matroid on $S$ obtained from $\phi$ of level $l$.  Then there is a 
canonical isomorphism
$$
\begin{CD}
\widetilde{H}_{l-1}(\ol{BC}(\bfM^*);\Bbbk)  
@>\theta >\cong >
T_S(\phi)
\end{CD}$$
which sends $$\ol{\sigma} _B \mapsto x_B^*\otimes e_S\otimes v_{S\backslash B}^*,$$
where $\{x_B^*\mid B\in \beta \mbox{nbc}(\bfM^*)\}$ is the basis of $S^*_S(\phi)$ dual to the basis $\{x_B\mid B\in \beta \mbox{nbc}(\bfM^*)\}$ of $S_S(\phi)$.
\end{lemma}

\begin{proof} If $\bfM$ is disconnected, then $T_S(\phi)=0$ and $\ol{BC}(\bfM^*)=\emptyset$ which gives that $\widetilde{H}_{l-1}(\ol{BC}(\bfM^*))=0$ unless $l=0$, in which case $\bfM$ is connected.  Therefore, we may assume $\bfM$ is a connected matroid.
We showed in Theorem \ref{T:MultiplicityBasis} 
that $\{x_B\mid B\in \beta \mbox{nbc}(\bfM^*)\}$ is a canonical basis of $S_S(\phi)$, so $\{x_B^*\otimes e_S\otimes v_{S\backslash B}^* \mid 
B\in \beta \mbox{nbc}(\bfM^*)\}$ is a canonical basis for $T_S(\phi)$.  Thus, both sets of basis elements are indexed by $\beta \mbox{nbc}(\bfM^*)$, hence $\theta$ is an isomorphism. 
\end{proof}

The remind the reader of the definition of the elements of the T-spaces from \cite{T}, in order to highlight the basis elements of the T-spaces defined in Lemma \ref{T:canonicalIso} by completing the example from Section \ref{BasisExample}.

\begin{example} \label{E:TchBasis}
In Section \ref{BasisExample}, we consider the free presentation of a multigraded module $L$

$$E\stackrel{\Phi}{\lra} G\lra L\lra 0,$$

given by $$\Phi = \left( \begin{array}{rrrrr}  x^3 & x^2y& xy^2& x^2 &y^3  \\ x^2 & 2xy & 3y^2 &x &0  \\ \end{array} \right)$$

with circuits $$\{\{ 1,4\}, \{ 1,2,3\}, \{ 1,2,5\}, \{ 1,3,5\}, \{ 2,3,4\}, \{ 2,3,5\}, \{ 2,4,5\}, \{ 3,4,5\}\}$$ and $\beta$-nbc bases $$\{\{1,4,5\}, \{1,3,4\}\}.$$

The basis elements of $S_{\{1,2,3,4,5\}}$ defined by the chain $\{1,4,5\}$ of the lattice of T-flats is given by $$x_{\{1,4,5\}}=\phi(e_4)\phi(e_5)=v_4\cdot v_5$$ and for $\{1,3,4\}$ by $$x_{\{1,3,4\}}=\phi(e_3)\phi(e_4)=v_3\cdot v_4,$$ since there is only one element removed as we descend the chain.

For computational purposes, it is often advantageous to work with $T^*_{\{1,2,3,4,5\}}(\phi)$ and then dualize rather than working directly with $T_{\{1,2,3,4,5\}}(\phi)$.  The basis for $T^*_{\{1,2,3,4,5\}}(\phi)$ is $$v_4\cdot v_5\otimes e^*_{12345} \otimes v_2\wedge v_3 \quad \mbox{ and }\quad v_3\cdot v_4\otimes e^*_{12345} \otimes v_2\wedge v_5.$$ 
The basis for $T_{\{1,2,3,4,5\}}(\phi)$ is dual to it given by $$(v_4\cdot v_5)^*\otimes e_{12345} \otimes v_2^*\wedge v_3^* \quad \mbox{ and }\quad (v_3\cdot v_4)^*\otimes e_{12345} \otimes v_2^*\wedge v_5^*.$$ 

For the last computation of $S_{12345}^*(\phi)$, we compute this term using divided powers, which is characteristic dependent.  The broken circuit basis, however did not need to specify a field at all.  In fact, all computations can be made over $\mathbb{Z}$, emphasizing the benefit of this new approach. 

\end{example}

\begin{theorem}\label{C:canonicalIso}
Let $\bfM$ be the ordered matroid obtained from $\phi$.  Then for each T-flat $A$ of level $l(A)$ there is a canonical isomorphism
$$\begin{CD}
\widetilde{H}_{l(A)-1}(\ol{BC}((\bfM|A)^*);\Bbbk) 
@>\theta |A >\cong>
T_A(\phi).
\end{CD}$$
\end{theorem}

\begin{proof}
This follows by considering the matroid $\bfM |A$ as in proof of Theorem \ref{C:BetaDim}.
\end{proof}

\begin{lemma}\label{L:ExactT-spaces}
Let $\bfM$ be a connected matroid on $S$ with representation $\phi$.  If $S$ and $S\backslash a$ are T-flats of $\bfM$ and $l=l(S)$, then the sequence of $\Bbbk$ vector spaces 
\[
\begin{CD}
0\rightarrow T_{S\backslash a}(\phi/a)
@>(-1)^{l }(\pi .S\backslash a)_{S\backslash a}^{\phi} >>
T_{S}(\phi)
@>(-1)^{|S\backslash a|}\phi_{l}^{S,S\backslash a} >>
T_{S\backslash a}(\phi) \rightarrow 0
\end{CD}
\]
is exact.
\end{lemma}

\begin{proof}
Although it was not stated explicitly, this was shown in \cite{T}.  The injective map is one component of the 
injective complexes of Theorem 7.7.  The vector space $T_{S\backslash a}(\phi)$ being the cokernal of $(-1)^{l }(\pi .S\backslash a)_{S\backslash a}^{\phi}$, denoted by 
$C_{l}$, 
is shown within the proof of Theorem 3.4, 3.5 and 5.45 in Section 9.  The author shows that the map 
$\gamma ^S_{l}: C_{l}\rightarrow T_{S\backslash a}(\phi)$ induced from $(-1)^{|S\backslash a|}\phi_n^{S,S\backslash a}
: T_S(\phi)\rightarrow T_{S\backslash a}(\phi)$ is an isomorphism.
\end{proof}

\begin{theorem}\label{T:CommDiagram}
Let $\bfM$ be an ordered matroid on $S$ with representation $\phi$ and $l=l(S)$.  If $a=\max(S)$ and $S\backslash a$ is a T-flat of $\bfM$, then the following diagram commutes,
\[
\begin{CD}
0 \rightarrow \widetilde{H}_{l-1}(\ol{BC}(\bfM^*\backslash a);\Bbbk)\stackrel{\tilde{\epsilon}}{\rightarrow}
@.
\widetilde{H}_{l-1}(\ol{BC}(\bfM^*);\Bbbk)\stackrel{\delta}{\rightarrow}
@.
\widetilde{H}_{l-2} (\ol{BC}(\bfM^*/a);\Bbbk) \rightarrow 0 
\\ 
@V \theta/a VV           
@V \theta VV   
@V \theta \backslash a VV  
\\ 
0\rightarrow T_{S\backslash a}(\phi/a)\lra
@.
T_{S}(\phi)\lra
@.
T_{S\backslash a}(\phi\backslash a)\rightarrow 0,
\end{CD}
\]
where the maps $\tilde{\epsilon}=(-1)^l \epsilon$ and $\delta$ are from Lemma \ref{L:HomologyMaps} and the bottom row is from Lemma \ref{L:ExactT-spaces}.
\end{theorem}

\begin{proof}
Notice that the top row of the diagram is not the sequence induced from the reduced chains of the reduced broken circuit complex, but since it is only a sign change, this sequence is also exact.
In order to show that our diagram is commutative, it is enough to show that the dual diagram 
\[
\begin{CD}
0 \rightarrow T_{S\backslash a}(\phi \backslash a)^* \rightarrow
@.
T_{S}(\phi)^*  \rightarrow
@.
T_{S\backslash a}(\phi/a)^* \rightarrow 0 
\\ 
@V  VV           
@V  VV   
@V  VV  
\\ 
0 \rightarrow \widetilde{H}_{l-2}(\ol{BC}(\bfM^*/ a) ;\Bbbk)^*\stackrel{\delta ^*}{\rightarrow}
@.
\widetilde{H}_{l-1}(\ol{BC}(\bfM^*);\Bbbk) ^*\stackrel{\tilde{\epsilon}^*}{\rightarrow}
@.
\widetilde{H}_{l-1} (\ol{BC}(\bfM^*\backslash a);\Bbbk)^* \rightarrow 0 
\end{CD}
\]
is commutative.  We will do this by investigating the images of the basis elements indexed by $\beta \mbox{nbc}(\bfM^*)$ under each map.  

Let $\{\ol{\sigma}_B^* |B\in \beta \mbox{nbc}(\bfM^*)\}$ in $\widetilde{H}_{l-1}(\ol{BC}(\bfM^*);\Bbbk)^*$ be the basis dual to $\{\ol{\sigma}_B |B\in \beta \mbox{nbc}(\bfM^*)\}$.  By taking duals, Lemma \ref{L:HomologyMaps} shows that 
$$\delta^*(\ol{\sigma}_B^*)=\ol{\sigma}_{B\cup a}^*$$
and $$(-1)^l \epsilon ^*(\ol{\sigma}_B^*)=\left\{ \begin{array}{ll} (-1)^l \ol{\sigma}_B^* & \mbox{ if } a\notin B\\
0 & \mbox{ otherwise.}\end{array}\right.$$
Also, Theorem \ref{T:canonicalIso} gives 
$$\theta ^*(x_B\otimes e_S^*\otimes v_{S\backslash B})=\ol{\sigma}_B^*,$$  
$$(\theta\backslash a) ^*(x_B\otimes e_{S\backslash a}^*\otimes v_{(S\backslash a)\backslash B})=\ol{\sigma}_B^*,$$
and 
$$(\theta/a) ^*(x_B\otimes e_{S\backslash a}^*\otimes \ol{v}_{(S\backslash a)\backslash B})=\ol{\sigma}_B^*.$$ 
Therefore, it remains to show that $$(-1)^{|S|-1}(\phi^{S, S\backslash a}_l)^*(x_B\otimes e_{S\backslash a}^* \otimes v_{(S\backslash a)\backslash B})=x_{B\cup a} \otimes e_S^* \otimes v_{S\backslash (B\cup a)}$$ 
and 
$$(-1)^l(\pi_{S\backslash a}^{\phi})_{S\backslash a}^*(x_B\otimes e_S^* \otimes v_{S\backslash B})=\left\{ \begin{array}{ll} (-1)^l x_B\otimes e_{S\backslash a}^*\otimes \ol{v}_{(S\backslash a)\backslash B} & \mbox{ if } a\notin B\\ 0& \mbox{ otherwise.} \end{array} \right.$$
Since $\bfM$ is connected, $V_S=V_{S\backslash a}$ and thus in \cite[Definition 2.3.1]{T} we have that $K_{S, S\backslash a}=0$.  Therefore, the map $\phi^{S, S\backslash a}_l$ from \cite[Definition 2.3.2]{T} can be rewritten as the composition
\[
\begin{CD}
T_S=S_S^*\otimes \twedge U_S\twedge V_S^*
\\
@VV \Delta \otimes \bfd \otimes 1 V
\\
V_a^*\otimes S_{S\backslash a}^* \otimes U_a \otimes \twedge U_{S\backslash a} \otimes \twedge V_{S\backslash a}^*
\\
@VV \tau V
\\
V_a^* \otimes U_a \otimes S_{S\backslash a}^* \otimes \twedge U_{S\backslash a} \otimes \twedge V_{S\backslash a}^*
\\
@VV 1 \otimes \phi \otimes 1 V
\\
V_a^* \otimes V_a \otimes S_{S\backslash a}^* \otimes \twedge U_{S\backslash a} \otimes \twedge V_{S\backslash a}^*
\\
@VV \mu \otimes 1 V
\\
S_{S\backslash a}^* \otimes \twedge U_{S\backslash a} \otimes \twedge V_{S\backslash a}^*=T_{S\backslash a}\\
\end{CD}
\]
The next diagram gives the composition of the dual maps in the left sequence and the image of $x_B\otimes e^*_{S\backslash a} \otimes v_{S\backslash B}$ in the right sequence.
\[
\begin{CD}
S_{S\backslash a}\otimes \twedge U^*_{S\backslash a} \twedge V_{S\backslash a} @. \quad \quad x_B\otimes e^*_{S\backslash a} \otimes v_{S\backslash B}
\\
@VV \mu ^* \otimes 1 V 
@VVV
\\
V_a\otimes V_a^*\otimes  S_{S\backslash a} \otimes \twedge U^*_{S\backslash a} \otimes \twedge V_{S\backslash a} @. \quad \quad v_a \otimes v_a^* \otimes x_B\otimes e^*_{S\backslash a} \otimes v_{S\backslash B}
\\
@VV 1\otimes \phi ^* \otimes 1 V 
@VVV
\\
V_a \otimes U_a^* \otimes S_{S\backslash a} \otimes \twedge U^*_{S\backslash a} \otimes \twedge V_{S\backslash a} @. \quad \quad v_a\otimes e_a^* \otimes x_B\otimes e^*_{S\backslash a} \otimes v_{S\backslash B}
\\
@VV \tau ^* V 
@VVV
\\
V_a \otimes S_{S\backslash a} \otimes U_a^* \otimes \twedge U^*_{S\backslash a} \otimes \twedge V_{S\backslash a} @. \quad \quad v_a \otimes x_B\otimes e_a^* \otimes e^*_{S\backslash a} \otimes v_{S\backslash B}
\\
@VV \nu \otimes \wedge \otimes 1 V 
@VVV
\\
S_S \otimes \twedge U^*_S \otimes \twedge V_S @. \quad \quad (-1)^{|S\backslash a|} x_{B\cup a}\otimes e^*_S \otimes v_{S\backslash B}\\
\end{CD}
\]

Therefore, $$(-1)^{|S|-1}(\phi^{S, S\backslash a}_l)^*(x_B\otimes e_{S\backslash a}^* \otimes v_{(S\backslash a)\backslash B})=x_{B\cup a} \otimes e_S^* \otimes v_{S\backslash (B\cup a)},$$ so that the left square commutes.

The map $$(\pi.S\backslash a)_{S\backslash a}^\phi :T_{S\backslash a}(\phi/a)\to T_S(\phi)$$ in \cite[Definition 7.3]{T} is given by the composition of the sequence of maps

\[
\begin{CD}
T_{S\backslash a}(\phi/a)=S_{S\backslash a}(\phi/a)^*\otimes \twedge U_{S\backslash a}\twedge V_{S\backslash a}(\phi/a)^*
\\
@VV 1 \otimes \bfdd  V
\\
S_{S\backslash a}(\phi/a)^* \otimes \twedge U_{S\backslash a} \otimes \twedge V_{S\backslash a}(\phi/a)^* \otimes U_a \otimes V_a^*
\\
@VV \tau V
\\
S_{S\backslash a}(\phi/a)^*\otimes \twedge U_{S\backslash a} \otimes U_a \otimes \twedge V_{S\backslash a}(\phi/a)^* \otimes V_a^*
\\
@VV \pi^* \otimes \wedge \otimes \bfc V
\\
S_S(\phi)^* \otimes \twedge U_S\otimes \twedge V_S(\phi)^*=T_S(\phi).\\
\end{CD}
\]
Similarly, we obtain image of a basis element $x_B\otimes e_S^* \otimes v_{S\backslash B}$ by composition of the sequence of dual maps as follows.
\[
\begin{CD}
 x_B\otimes e^*_S \otimes v_{S\backslash B}
\\
@VV \pi \otimes \bfd \otimes \bfc ^* V 
\\
 \ol{x_B}\otimes e_{S\backslash a}^* \otimes e_a^* \otimes \ol{v}_{(S\backslash B)\backslash a} \otimes v_a
\\
@VV \tau ^* V 
\\
\ol{x_B}\otimes e_{S\backslash a}^* \otimes \ol{v}_{(S\backslash B)\backslash a}\otimes e_a^* \otimes v_a
\\
@VV 1\otimes \bfdd ^* V 
\\
\ol{x_B} \otimes e_{S\backslash a}^* \otimes \ol{v}_{(S\backslash a)\backslash B}\\
\end{CD}
\]

Therefore, $$(-1)^l(\pi_{S\backslash a}^{\phi})_{S\backslash a}^*(x_B\otimes e_S^* \otimes v_{S\backslash B})=\left\{ \begin{array}{ll} (-1)^l x_B\otimes e_{S\backslash a}^*\otimes v_{(S\backslash a)\backslash B} & \mbox{ if } a\notin B\\ 0& \mbox{ otherwise} \end{array} \right.$$ and the diagram commutes.
\end{proof}

\end{document}